\documentclass[10pt,a4paper,oneside,reqno,notitlepage]{amsart}
\usepackage{amssymb,color}
\usepackage{graphicx}
\usepackage[arrow]{xy}
\usepackage{pb-diagram,pb-xy}
\usepackage{hyperref}
\usepackage{url}

\usepackage{amsfonts}
\usepackage{amscd,amsmath}
\usepackage{euscript}

\theoremstyle{plain}
\newtheorem{theorem}{Theorem}[section]
\newtheorem{prop}[theorem]{Proposition}

\newtheorem{lemma}[theorem]{Lemma}
\newtheorem{remark}[theorem]{Remark}
\newtheorem{example}[theorem]{Example}

\newtheorem{ques}[theorem]{Question}

\newtheorem{FP}[theorem]{Fiber-Product Theorem}
\newtheorem*{main}{Theorem~\ref{thm:main}}

\theoremstyle{definition}
\newtheorem{defi}[theorem]{Definition}

\newtheorem*{note*}{Notation}

\newcommand{\rst}[1]{\ensuremath{{\mathbin\upharpoonright}%
\raise-1ex\hbox{$\small{#1}$}}}

\def\Z{\mathbb{Z}}

\def\E{\mathcal E}
\def\F{\mathcal F}

\begin{document}
\title{Group localization and two problems of Levine}

\author{Roman Mikhailov}
\address{Chebyshev Laboratory, St. Petersburg State University, 14th Line, 29b,
Saint Petersburg, 199178 Russia and St. Petersburg Department of
Steklov Mathematical Institute} \email{rmikhailov@mail.ru}
\urladdr{http://www.mi.ras.ru/\~{}romanvm/pub.html}
\thanks{This research of the first author is supported by the Chebyshev
Laboratory  (Department of Mathematics and Mechanics, St.
Petersburg State University)  under RF Government grant
11.G34.31.0026}

\author{Kent E. Orr}
\address{Dept of Mathematics, Indiana University, Bloomington IN 47405}
\email{korr@indiana.edu}
\thanks{The Simons Foundation generously supports the second author through grant 209082.  These results were obtained while the second author enjoyed the city of St. Petersburg and the support of the Chebyshev Laboratory, St. Petersburg State University.  This author appreciates their support and the excellent research atmosphere provided.}

\maketitle

\begin{abstract}
A. K. Bousfield's $H\Z$-localization of groups inverts homologically two-connected homomorphisms of groups.  J. P. Levine's algebraic closure of groups inverts homomorphisms between finitely generated and finitely presented groups which are homologically two-connected and for which the image normally generates.  We resolve an old problem concerning Bousfield $H\Z$-localization of groups, and answer two questions of Levine regarding algebraic closure of groups.  In particular, we show that the kernel of the natural homomorphism from a group $G$ to it's Bousfield $H\Z$-localization is not always a $G$-perfect subgroup.  In the case of algebraic closure of groups, we prove the analogous result that this kernel is not always an invisible subgroup.
\end{abstract}

\section{introduction}


In~\cite{Bousfield:1974-1,Bousfield:1975-1,Bousfield:1977}, A. K. Bousfield defines his homology localization of spaces, and initiates an investigation of a corresponding localization of groups and modules.  Roughly, the $H\Z$-localization of spaces is the initial functor inverting homology equivalences of spaces. If $E(G)$ is the group localization of a group $G$, and $E(X)$ the homology localization of a space $X$, then $E(\pi_1(X)) \cong \pi_1(E(X))$.  Homology localization of groups inverts homomorphisms that are isomorphisms on abelianization and epimorphisms on second homology.

In analogy with Bousfield's $H\Z$-localization of spaces, P. Vogel, in unpublished work, defines a localization of spaces which inverts maps between finite complexes $X \to Y$ with a contractible cofiber. The primary significance of Vogel's work arises from its connection to homology cobordism classes of manifolds and embedding theory. (See, for example, \cite{Cha:2004-1,Cha-Orr:2010-02,Heck:2010-01,LeDimet:1988-1,Levine:1988-1,Levine:1989-1,Levine:1989-2,Levine:1994-1,Sakasai:2006-1}.)

In~\cite{Levine:1989-1}, J. P. Levine defines his algebraic closure of groups.  (An earlier version of this type of closure was defined by M. Gutierrez~\cite{Gutierrez:1979-1}.)  Similarly to the connection between Bousfield's $H\Z$-localization of groups and spaces, the algebraic closure of a group $G$ describes how the Vogel localization of a space effects the fundamental group of the space.  We denote both Bousfield's $H\Z$-closure of the group $G$ and Levine's algebraic closure of $G$ by $E(G)$. 

The natural homomorphism $G \to E(G)$ rarely embeds $G$.  In particular, for $H \Z$-localization of
groups, Bousfield defines a family of subgroups which always lie in the kernel of the natural homomorphism $G \to E(G)$.  A normal subgroup $N \mathrel{\unlhd} G$ is called a {\em $G$-perfect subgroup} if $N = [G, N]$.  In this study of algebraically closed groups, Levine's adds the condition that the $G$-perfect subgroup is normally generated by a finite subset of group elements, and calls these groups {\em invisible groups.}

We replace the terms $G$-perfect and invisible with the term {\em relatively perfect} subgroup.  One can easily tell from context if relatively perfect means Bousfield's $G$-perfect subgroup, or Levine's invisible subgroup.
 
Bousfield, (respectively Levine) prove that $H\Z$-local groups (respectively, algebraically closed groups) do not contain non-trivial relatively perfect subgroups.

{\em In particular, for any group $G$, all relatively perfect
subgroups of $G$ lie in $\ker\{ G \to E(G)\}$.}

\begin{ques}\label{ques:kernel}
What is the $\ker\{G \to E(G)\}$?  More precisely, we ask this question.  For Bousfield's
$H\Z$-localization of  a group $G$, is this kernel a relatively
perfect subgroup? For Levine's group closure, is this kernel the
union of the relatively perfect subgroups?
\end{ques}

Progress has been negligible, and constructions that might produce
interesting examples rare. Bousfield does not appear to have
explicitly asked the question in any paper.   However, the problem
has long circulated the field.  J. Rodriguez and D. Scenvenls gave equivalent formulations for this problem, without resolution~\cite{Rodriguez-Scevenels:2004-1}.

In~\cite{Levine:1990-1}, Question 4, Levine explicitly asks the above
question. He also asks, in Question 5 of that paper, whether the
kernel of a homologically $2$-connected homomorphism $G \to
\Gamma$ is a union of relatively perfect subgroups when $G$ is
finitely generated.

The main Theorem of this paper gives a negative answer
Question~\ref{ques:kernel}, including Levine's Question~$4$ and~$5$
from~\cite{Levine:1990-1}.  To answer this question we apply a new construction arising from our extension of a result by M. Bridson and A. Reid~\cite{Bridson-Reid:2012-1}, which we call the {\em Fiber-Product Theorem.}  We state our main theorem below.  The Fiber-Product Theorem is stated and proven in Section~\ref{FPTheorem}.

\begin{main}
Let $Q$ be any finitely presented perfect group with a balanced presentation, such that $H_3(Q) \not \cong 0$. Suppose $Q \cong F/R$ where $F$ is free on the generators in the balanced presentation of $Q$.  Consider the fiber product, $F \times_Q F \cong \E/K$, where $\E$ is free on the generating set of $Q$.  Let $G := \E/[\E,K]$ be the associated free central extension of $\E/K$.  Then
\begin{enumerate}
\item $G$ is a finitely generated group,
\item $\gamma_{\omega}(G)\neq \gamma_{\omega +1}(G)=\{1\}$,\item and if $E(G)$ represents either the Bousfield $H \Z$-localization of $G$, or if $E(G)$ is Levine's algebraic closure of $G$ and in
addition in this case, $Q$ is a finite group, then
\[
\gamma_{\omega}(G) = \ker\{G \to E(G)\}.
\]
In particular, in both cases, the $\ker\{G \to E(G)\}$ is
not a relatively perfect subgroup of $G$.
\end{enumerate}
\end{main}

We provide examples of finite, and finitely presented, perfect groups, as used in Theorem~\ref{thm:main}, within the paper.  We note, for now, the particularly interesting example, the {\em Binary Icosahedral group}, sometimes called the {\em Poincar\'e group}, that is, the fundamental group of the Poincar\'e homology sphere.

\section{$H\Z$-localization and algebraic closure of groups}

\subsection{$H\Z$-localization of groups} We discuss, briefly, the $H\Z$-localization and Levine closure of groups, offering a little more detail than in the introduction. We recommend~\cite{Bousfield:1974-1,Bousfield:1975-1} and  \cite{Levine:1989-1} for a careful treatment of these topics.

\begin{defi} A homomorphism of groups $G \to \Gamma$ is called {\em homologically $2$-connected} if the induced homomorphism on homology, $H_{k}(G) \to H_{k}(\Gamma)$, is an isomorphism for $k = 1$ and onto for $k = 2$.
\end{defi}

\begin{defi}\label{defi:local group} A group $\Gamma$ is {\em $H\Z$-local} if given any diagram of homomorphisms as follows, with $G \to H$ homologically $2$-connected,
\[
\begin{diagram}
\node{G} \arrow{e}\arrow{se} \node{H}\arrow{s,..}\\
\node[2]{\Gamma}
\end{diagram}
\]
there is a {\em unique} homomorphism $H \to \Gamma$ making the diagram commute.
\end{defi}

\begin{defi}\label{defi:localization}  The $H\Z$-localization of a group $G$ is a functor $E$ which takes groups to $H\Z$-local groups, and for each group $G$ a natural, homologically $2$-connected, homomorphism $G \to E(G)$ such that given any homomorphism $G \to \Gamma$ with $\Gamma$ an $H\Z$-local group, there is a unique homomorphism $E(G) \to \Gamma$ making the following diagram commute.
\[
\begin{diagram}
\node{G} \arrow{e} \arrow{se} \node{E(G)}\arrow{s,..}\\
\node[2]{\Gamma}
\end{diagram}
\]
\end{defi}
That is, the $H\Z$-localization is the initial local group determined by the group $G$.

Bousfield constructs and investigates the unique (up to natural equivalence of functors) $H\Z$-localization of groups.  In particular, he shows that if $X$ is a space with $\pi_{1}(X) \cong G$, then $\pi_{1}(E(X)) \cong E(G)$, where $E(X)$ is the $H\Z$-localization of the space $X$.

\subsection{Levine's algebraic closure of groups}

\begin{defi}
A group $\Gamma$ is {\em algebraically closed} if given any homologically $2$-connected homomorphism $G \to H$ where $G$ is finitely generated and $H$ is finitely presented, and such that the image of $G \to H$ normally generates $H$, and a homomorphism $G \to \Gamma$, there is a unique extension $H\to \Gamma$ making the diagram in Definition~\ref{defi:local group} commute.
\end{defi}

Levine defines closed groups using systems of equations over the group, but observes it satisfies the above properties.  

As in Definition~\ref{defi:localization}, Levine's algebraic closure of groups is a functor, $E$, from groups to algebraically closed groups together with a natural homomorphism for each group $G$, $G \to E(G)$, which is initial among homomorphisms to algebraically closed groups.

We mention briefly that the $HR$-localization of a group has also been described as an algebraic closure of groups~\cite{Farjoun-Orr-Shelah:1989}.  Solving equations over groups plays no role in the theorems proven herein.

\subsection{Relatively perfect subgroups}

\begin{defi}
A subgroup $N \mathrel{\unlhd} G$ is {\em relatively perfect} if
$N = [G,N]$.
\end{defi}

Levine calls a group $N \mathrel{\unlhd} G$ {\em invisible} if $N$
is relatively perfect and, additionally, is  normally generated in
$G$ by a finite set of elements.  We avoid the term {\em
invisible} and when working in the context of Levine's group
closure we call these groups relatively perfect as well.
Whether relatively perfect includes the normal generation
condition is clear from context, depending on whether we are
discussing Bousfield $H \Z$-localization, or Levine's group
closure.  That is, the notion we call relatively perfect, in the
context of algebraic closure of groups will additionally require
that $N$ is normally generated by a finite collection of elements.

Bousfield, and separately Levine, prove this theorem~\cite[Proposition~1.2]{Bousfield:1977}\cite[Lemma~3]{Levine:1989-1}.

\begin{theorem}[Bousfield, Levine]
If $G$ is an $H \Z$-local (resp. algebraically closed group) and
$N \mathrel{\unlhd} G$ is a relatively perfect subgroup, then $N
\cong \{1\}$.

In particular, for any group $G$, all relatively perfect
subgroups of $G$ lie in $\ker\{ G \to E(G)\}$.
\end{theorem}

\section{Two lemmas: group homology and the lower central series}

This section contains the statement and proof of two results (of Emmanouil-Mikhailov and Mikhailov-Passi, respectively) fundamental to proving the main results of this paper.  We include these short proofs to assure completeness and readability.

We first establish notation for elementary spectral sequence arguments.

For a group extension
\[
1 \to R \to G \to Q \to 1
\]
and any $\Z Q$-module $M$, the Hochschield-Serre spectral sequence gives a filtration
\[
\{0\} \subset D_{0,n} \subset D_{1,n-1} \subset \cdots \subset D_{n,0} = H_n(G; M)
\]
with associated graded
\[
E_{p,q}^{\infty} \cong D_{p,n-p}/D_{p-1,n-p+1}
\]
and
\[
E_{p,q}^2 = H_{p}(Q, H_{q}(R; M)) \Rrightarrow H_{p+q}(G; M).
\]

\begin{defi}\label{defi:super-perfect}
A group $G$ is {\em super-perfect} if $H_{1}(G) \cong H_{2}(G) \cong 0$.
\end{defi}

\begin{example}
The Binary Icosahedral group
\[
Q := \langle a, b \quad \vert \quad a^5=b^3,\ b^3 = (ba)^2 \rangle
\]
is an order $120$ super-perfect group that arises as the fundamental group of the Poincar\'e sphere, a compact, orientable $3$ manifold with perfect fundamental group.
\end{example}

For the above elegant and geometric example, $H_3(Q)$ is cyclic of order $120$.  This group suffices to construct examples illustrating each of the main theorems of this paper.

We emphasize that super-perfect groups are plentiful.  Indeed, suppose that a finitely presented group $G$ has a balanced presentation, that is, the number of generators and relations are the same and finite.  If $H_{1}(G) \cong 0$ then one constructs, using the presentation for $G$, a two-dimensional CW-complex $X$ with $\pi_{1}(X) \cong G$ and having a cell of dimension one and two for each generator of $G$ in the presentation.  This implies that $H_{2}(X) \cong 0$, and since $X$ is the $2$-skeleton for a $K(\pi_{1}(X), 1) = K(G,1)$ complex, it follows that $H_{2}(G) \cong 0$.

The following is a special case of Proposition~$1.25$ from~\cite{Emmanouil-Mikhailov:2008}.

\begin{lemma}\label{lemma:EM}
Suppose we have a group presentation
\[
1 \to R \to F \to Q \to 1
\]
where $F$ is a free group.  Let $G$ be the semi-direct product, $G := R \rtimes F$, where $F$ acts on $R$ by conjugation.  Then there is an exact sequence
\[
H_{4}(Q) \to H_{0}(Q; R_{ab}\otimes R_{ab}) \to H_{2}(G) \to H_{3}(Q) \to 0.
\]
\end{lemma}

\begin{proof}
First note that $H_{2}(G) \cong H_{1}(F; R_{ab})$.  Here $R_{ab}$
is the abelianization of the group of relations $R$, which can be
viewed as a $\Z Q$-module, where $Q$ acts via conjugation in $F$.
One easily sees the above isomorphism using the spectral sequence
for the extension $1 \to R \to G \to F \to 1$. More precisely
since $F$ and $R$ are free groups,
\[
H_{2}(G) \cong H_{2}(R \rtimes F) = D_{2,0} = E_{1,1}^{2} = H_{1}(F; H_{1}(R)) = H_{1}(F; R_{ab}).
\]

Similarly, using the spectral sequence of the extension $F \to Q$, and that $R$ is free, we observe that for all $p \geq 1$,
\[
H_{p+2}(Q) \cong H_{p}(Q; H_{1}(R)).
\]

Now consider the homology spectral sequence with $R_{ab}$ coefficients associated to
\[
1 \to R \to F \to Q \to 1.
\]
Then
\[
E_{p,0}^{2} = H_{p}(Q; H_{0}(R; R_{ab})) \cong H_{p}(Q; R_{ab}) \cong H_{p+2}(Q), \quad p \geq 1.
\]
Here the second isomorphism follows since $R$ acts on $R$ by
conjugation, and so acts trivially on $R_{ab}$. The last
isomorphism is the observation made in the previous paragraph.
\[
E_{p,q}^{2} = H_{p}(Q; H_{q}(R; R_{ab})) \cong 0, \quad q \geq 2,
\]
since $R$ is free.  Note as well that $E_{1,0}^{\infty} = E_{1,0}^2 \cong H_3(Q)$.

For $q = 1$,
\[
E_{p,1}^{2} = H_{p}(Q ; H_{1}(R; R_{ab})) \cong H_{p}(Q; R_{ab} \otimes R_{ab}).
\]
Hence we have the usual exact sequence of low order terms
\[
E_{2,0}^{2} \xrightarrow{d_{2,0}^{2}} E_{0,1}^{2} \to D_{1,0} \to E_{1,0}^{\infty} \to 0
\]
and using the above computations, this completes the proof.
\end{proof}

Recall the transfinite lower central series of a group $G$.

\begin{defi} For a discrete ordinal $k$, the {\em $k^{th}$ lower central series subgroup of $G$} is defined by $\gamma_{1}(G) = G$, and $\gamma_{k+1}(G) = [G, \gamma_{k}(G)]$, where for normal subgroups $H, K \mathrel{\triangleleft} G$, the symbol $[H, K]$ denotes the subgroup generated by commutators $[h, k] := h^{-1}k^{-1}hk \in G.$  For $\lambda$ a limit ordinal,
\[
\gamma_{\lambda}(G) = \bigcap_{\delta < \lambda}\gamma_{\delta}(G).
\]
\end{defi}

\begin{defi}
A group is residually nilpotent if $\gamma_{\omega}(G) = \cap_{k} \gamma_{k}(G) = 1$.
\end{defi}

Recall as well the Dwyer filtration on the second homology of a group $G$~\cite{Dwyer:1975-1}.

\begin{defi}
For $K$ finite, the $k^{th}$ term in the {\em Dwyer filtration} of the second homology of a group $G$ is defined as
\[
\phi_{k}(G) := \ker\{H_{2}(G) \to H_{2}(G/\gamma_{k-1}(G)\},\
k\geq 2\] and the first transfinite term by
\[
\phi_{\omega}(G) = \cap_{k}\phi_{k}(G).
\]
\end{defi}

The next Lemma appears in~\cite{Mikhailov-Passi:2006-1}.

\begin{lemma}\label{lemma:MP}
Let $G \cong F/R$ where $F$ is a free group.  If $G$ is residually nilpotent, then
\[
\gamma_{\omega}\left(\frac{F}{[F,R]}\right) \cong \phi_{\omega}(G)
\]
\end{lemma}

\begin{proof}
We first compute the $k^{th}$ term in the Dwyer filtration of $H_{2}(G), k$ finite.  Note that
\begin{align}
&H_{2}(G) \cong \frac{R \cap [F,F]}{[F,R]}, \text{ and }\\
&H_{2}\left(\frac{G}{\gamma_{k}(G)}\right) \cong H_{2}\left(\frac{F}{\gamma_{k}(F)\cdot R}\right) \cong \frac{\gamma_{k}(F)\cdot R \cap [F,F]}{[F,\gamma_{k}(F)\cdot R]} = \frac{\gamma_{k}(F)\cdot R \cap [F,F]}{\gamma_{k+1}(F)\cdot [F,R]},
\end{align}
where the last isomorphism follows since $\gamma_k(F)$ and $R$ are normal subgroups of $F$.
Combining the equations (1) and (2) above,
\begin{align*}
\phi_{k+1}(G) &\cong \ker\left\{\frac{R \cap [F,F]}{[F,R]} \to
\frac{\gamma_{k}(F)\cdot R \cap [F,F]}{\gamma_{k+1}(F)\cdot
[F,R]}\right\}\\ = &\frac{R \cap \gamma_{k+1}(F) \cdot
[F,R]}{[F,R]} \subset \frac{\gamma_{k+1}(F) \cdot [F,R]}{[F,R]} =
\gamma_{k+1}\left(\frac{F}{[F,R]}\right).
\end{align*}
Therefore
\[
\phi_{\omega}(G) = \bigcap_{k}\phi_{k}(G) \subset \gamma_{\omega}(G).
\]

To prove the other inclusion, we use the hypothesis that $G$ is residually nilpotent.

There is a quotient homomorphism $\frac{R}{[F,R]} \to
\frac{F}{R}$.  Given $g\cdot [F,R] \in
\gamma_{\omega}\left(\frac{R}{[F,R]}\right)$ it follows that $g\cdot
R \in \gamma_{\omega}\left(\frac{F}{R}\right)$, and since $F/R$ is
residually nilpotent, $g \in R$. Thus we have the following two relations.

\begin{align*}
\gamma_{\omega}\left(\frac{F}{[F,R]}\right) \subset \frac{R}{[F,R]} \qquad \qquad \gamma_{\omega}\left(\frac{F}{[F,R]}\right)  = \bigcap_{k}\frac{\gamma_{k}(F)\cdot[F,R]}{[F,R]}\\
\end{align*}
So for any $k$,
\[
\gamma_{\omega}\left(\frac{F}{[F,R]}\right) \subset
\frac{R}{[F,R]} \cap  \frac{\gamma_{k+1}(F)\cdot[F,R]}{[F,R]} =
\frac{R \cap \gamma_{k+1}{(F)}\cdot [F,R]}{[F,R]} = \phi_{k+1}(G),
\]
where the middle equality follows since $[F,R]$ is normal in $R$ and in $ \gamma_{k+1}(F)\cdot [F,R]$.\

Thus, $\gamma_{\omega}\left(\frac{F}{[F,R]}\right) \subset \phi_{\omega}(G)$, proving the Lemma.
\end{proof}

\section{The Fiber-Product Theorem}\label{FPTheorem}
We now prove a theorem which we view as an extension of Propositions~$3.1$ and $3.2$ of Bridson and Reid~\cite{Bridson-Reid:2012-1}.  Most notably, we make no assumptions about finite index subgroups.

We offer a further extension in a future paper with G. Baumslag, whose proof uses simplicial methods.  In spite of the elementary proof given below, this result underpins the remainder of the paper.

\begin{FP}\label{theorem:FP}
Suppose $Q$ is a super-perfect group with the following presentation, where $F$ is a free group.
\[
1 \to R \to F \to Q \to 1
\]
Consider the following pull-back diagram
\[
\begin{diagram}
\node{P} \arrow{e}\arrow{s} \node{F}\arrow{s,r}{p}\\
\node{F} \arrow{e,t}{p} \node{Q}
\end{diagram}
\]
and denote $P := F \times_{Q} F$, that is, $P = \{(x,y) \quad \vert \quad p(x) = p(y)\}$ is the fiber product over $p \colon F \twoheadrightarrow Q$.

Then,
\begin{enumerate}
\item The inclusion homomorphism
$F \times_{Q} F \to F \times F$
is homologically $2$-connected.
\item If $\ker\{H_{2}(F\times_{Q} F) \to H_{2}(F \times F)\} \cong 0$ then $H_{3}(Q) \cong 0$.
\end{enumerate}
\end{FP}

\begin{proof}
The homomorphism $F \times_Q F \to F$ splits since $F$ is free, and has kernel $R = \ker\{F \to Q\}$ since $P = F \times_Q F$ is a pullback with fiber $R$.  Thus, $F \times_Q F \cong R \rtimes F$.  Hence,
\[
H_1(F\times_Q F) \cong H_1(R \rtimes F) \cong R/[F,R] \oplus F_{ab},
\]
where $F_{ab}$ is the abelianization of the free group $F$.  We
have a commutative diagram where the bottom horizontal
homomorphism is a sum of a homomorphism induced by the inclusion
$inc \colon R \to F$ and the identity.
\[
\begin{diagram}
\node{H_1(F \times_Q F)} \arrow[3]{e}\arrow{s,r}{\cong} \node[3]{H_1(F \times F)}\arrow{s,r}{\cong}\\
\node{R/[F,R] \oplus F_{ab}} \arrow[3]{e,t}{(inc_*)\oplus id}\node[3]{F_{ab} \oplus F_{ab}}
\end{diagram}
\]
The extension $R \rightarrowtail F \twoheadrightarrow Q$ determines the usual $5$-term exact sequence on homology groups,
\[
0 \to H_2(Q) \to R/[F,R] \xrightarrow{inc_*} F_{ab} \to H_1(Q) \to 0.
\]
Combining this with the above commutative square of homomorphisms, we get the following elegant exact sequence:
\begin{equation}\label{exact sequence}
0 \to H_2(Q) \to H_1(F \times_Q F) \xrightarrow{inc} H_1(F \times F) \to H_1(Q) \to 0.
\end{equation}
Thus, $H_1(F \times_Q F) \to H_1(F \times F)$ is an isomorphism if and only if $Q$ is super-perfect, that is, $H_1(Q) \cong H_2(Q) \cong 0$.

To compute the image of the second homology, the homomorphism of presentations
\[
\begin{diagram}
\node{1} \arrow{e} \node{R} \arrow{e}\arrow{s} \node{F} \arrow{e}\arrow{s} \node{Q} \arrow{e} \arrow{s} \node{1}\\
\node{1} \arrow{e} \node{F} \arrow{e,t}{=} \node{F} \arrow{e} \node{1}
\end{diagram}
\]
yields a homomorphism of exact sequences, as given in Lemma~\ref{lemma:EM}, as follows.
\[
\begin{diagram}
\node{H_4(Q)} \arrow{e} \node{H_0(Q; R_{ab}\otimes R_{ab})} \arrow{e}\arrow{s,A} \node{H_2(F \times_Q F)} \arrow{e,A} \arrow{s} \node{H_3(Q)}\\
\node[2]{H_0(1; F_{ab} \otimes F_{ab})} \arrow{e,t}{\cong}
\node{H_2(F \times F)}
\end{diagram}
\]

Since $R_{ab} \to F_{ab}$ is onto (that is, since $Q$ is perfect
and using exact sequence~(\ref{exact sequence}) given
above,) the left vertical homomorphism is onto as shown below.
\[
H_0(Q; R_{ab} \otimes R_{ab})\cong (R_{ab} \otimes R_{ab}) \otimes_{\Z Q} \Z \twoheadrightarrow F_{ab} \otimes F_{ab} \cong H_0(1; F_{ab} \otimes F_{ab})
\]
Thus, $H_2(F\times_Q F) \to H_2(F \times F)$ is onto as claimed.

To prove statement $(2)$, if $H_2(F\times_Q F) \to H_2(F \times
F)$ is $1$-$1$, then it is an isomorphism, and hence, $H_0(Q,
R_{ab} \otimes R_{ab}) \to H_2(F \times_Q F)$ is onto, implying
that $H_3(Q) \cong 0$ as claimed.
\end{proof}

\begin{remark}
Fibre Product Theorem \ref{theorem:FP} can be generalized to the
case of fibre products with many summands. For $n\geq 2$ and an
epimorphism $p: F\to Q$, consider the higher fibre product
$$
F_n(p):=\underbrace{F\times_Q\times\dots\times_QF}_n=\{(x_1,\dots,
x_n)\in F^{\times n}\ |\ p(x_1)=\dots=p(x_n)\}
$$
Then the natural inclusion
$$
F_n(p)\hookrightarrow F^{\times n}
$$
is homologically 2-connected provided $Q$ is super-perfect.
\end{remark}
The proof follows by induction on $n$, using the natural split
exact sequence, $(n\geq 3)$,
$$
1\to R\to F_n(p)\to F_{n-1}(p)\to 1.
$$
\section{Applications to $H\Z$-localization and closure of groups}

We begin by describing a new construction using the Fiber-Product Theorem~\ref{theorem:FP}.

Let $Q$ be a finitely presented perfect group with a balanced
presentation, that is $Q$ has a presentation with an equal number
of generators and relations. As observed following
Definition~\ref{defi:super-perfect}, this implies $Q$ is
super-perfect.

Let $F$ be the free group on the generating set for $Q$, with $F
\twoheadrightarrow Q$ the associated epimorphism.  Let $R =
\ker\{F \twoheadrightarrow Q\}$. As noted previously, $F \times_Q
F \cong R \rtimes F$ where $F$ acts on $R$ by conjugation.

Note that if $Q$ has $k$ generators, and therefore $k$ relations,
then $R$ is normally generated by $k$ elements, and the group $F
\times_Q F \cong R \rtimes F$ is generated by $2k$-elements.
Since $Q$ is super-perfect, $H_1(F \times_Q F) \cong H_1(F \times
F) \cong \Z^{2k}$ by statement~$(2)$ of the Fiber Product
Theorem~\ref{theorem:FP}, and so $F \times_Q F$ is generated by at
least $2k$ generators, and thus by exactly $2k$ generators.

Let $\E$ and $\F$ be free on $2k$ letters such that $F \times_Q F
\cong \E/K$, and $F \times F \cong \F/L$.  Since $\E$ is free, one
easily constructs a homomorphism $\E \to \F$ inducing the
following commutative diagram of groups.
\[
\begin{diagram}
\node{\frac{\E}{[\mathcal \E,K]}} \arrow{e}\arrow{s} \node{\frac{\F}{[\F,L]}}\arrow{s}\\
\node{\frac{\E}{K} \cong F \times_Q F}\arrow{e} \node{F \times F \cong \frac{\F}{L}}
\end{diagram}
\]
\begin{note*}
Let $Q$ be a perfect group with balanced presentation $\mathcal P$.  We denote the homomorphism above,
\[
\Phi_{(Q,\mathcal P)} \colon \frac{\E}{[\E,K]} \to \frac{\F}{[\F,L]}.
\]
\end{note*}

\begin{prop}\label{prop:free extension}
Let $Q$ be a perfect group with a balanced presentation $\mathcal P$, such that $H_3(Q) \not \cong 0$.  Then the following holds:
\begin{enumerate}
\item $\frac{\E}{[\E,K]}$ is a finitely generated group. \item
$\Phi_{(Q, \mathcal P)} \colon \frac{\E}{[\E,K]} \to
\frac{\F}{[\F,L]}$ is homologically $2$-connected. \item
$\gamma_{\omega}\left(\frac{\E}{[\E,K]}\right) \neq \{1\}$ and
$\gamma_{\omega +1}\left(\frac{\E}{[\E,K]}\right) = \{1\}$ \item
$\frac{\F}{[\F,L]}$ is residually nilpotent. \item If $Q$ is
finite, then $\frac{\E}{[\E,K]}$ is a finitely presented group.
\end{enumerate}
\end{prop}

\begin{proof}
To see $(1)$, $\E$ is free on twice the number of generators of $Q$ and thus $\E/[\E,L]$ is finitely generated.

To prove statement $(2)$, first note that since $Q$ is super-perfect, the homomorphism $H_1(F \times_Q F) \to H_1(F \times F)$ is an isomorphism by the Fiber Product Theorem~\ref{theorem:FP}, and these first homology groups have rank $2k$ since $F \times F$ has rank $2k$. Since $\E$ and $\F$ are free of rank $2k$ as well, it follows that $H_1(\E/[\E,K]) \to H_1(\F/[\F.L])$ is an isomorphism of rank $2k$ groups.

To prove $\Phi_{(Q, \mathcal P)}$ induces an epimorphism on second homology, we compare the spectral sequences associated to the extensions
\begin{equation}\label{extensions}
\frac{K}{[\E,K]} \rightarrowtail \frac{\E}{[\E,K]} \twoheadrightarrow \frac{\E}{K} \qquad \mbox{and} \qquad
\frac{L}{[\F,L]} \rightarrowtail \frac{\F}{[\F,L]} \twoheadrightarrow \frac{\F}{L}.
\end{equation}
For each of these spectral sequences the differential $d_{2,0} \colon E_{2,0}^2 \to E_{0,1}^2$ are onto since we have shown that the quotient homomorphisms in the extensions above are isomorphisms on abelianization. So we have an exact sequence
\[
E_{0,2}^{\infty} \to H_2 \to E_{1,1}^{\infty} \to \{0\}.
\]
So we have a homomorphism of short exact sequences, where $\frac{\Lambda^2(A)}{\sim}$ is the group of second exterior powers of $A$, that is $E_{2,0}^2$, modulo the image of differentials, and the right-most terms are quotients of $E_{1,1}^2$ by the image of $d_{3,0}^2$.
\[
\begin{diagram}
\node{\frac{\Lambda^2(K/[\E,K])}{\sim}} \arrow{e}\arrow{s,r}{\alpha} \node{H_2(\E/[\E,K])}\arrow{e,A}\arrow{s,r}{\Phi_{(Q, \mathcal P)_*}} \node{\frac{H_1(\E/K; K/[\E,K])}{\sim}}\arrow{s,r}{\beta}\\
\node{\frac{\Lambda^2(L/[\F,L])}{\sim}} \arrow{e} \node{H_2(\F/[\F,L])} \arrow{e,A}\node{\frac{H_1(\F/L; L/[\F,L])}{\sim}}
\end{diagram}
\]

By statement~$(1)$ of the Fiber Product Theorem~\ref{theorem:FP}, and since $H_2(\E/K) \cong K/[\E,K]$ and $H_2(\F/L) \cong L/[\F,L]$, it follows that the homomorphism $K/[\E,K] \to L/[\F,L]$ is onto.  Hence $\alpha$ is onto in the  above diagram.

Since the displayed extensions above, numbered~(\ref{extensions}), are central extensions, $H_1(\E/K; K/[\E,K]) \cong H_1(\E/K) \otimes K/[\E,K]$, and $H_1(\F/L; L/[\F,L]) \cong H_1(\F/L) \otimes L/[\F,L]$.  Since $H_1(\E/K) \cong H_1(\F/L)$, the homomorphism $\beta$ from the above diagram is onto as well.

Thus,
\[
\Phi_{(Q, \mathcal P)_*} \colon H_2(\E/[\E,K]) \to H_2(\F/[\F,L])
\]
is onto as claimed, proving statement $(2)$ of Proposition~\ref{prop:free extension}.

To prove statement $(3)$, first observe that $F \times_Q F \subset
F \times F$. The latter is residually nilpotent since the free
group $F$ is residually nilpotent.  Hence, $\E/K \cong F \times_Q
F$ is residually nilpotent. By Lemma~\ref{lemma:MP},
$\gamma_{\omega}(\E/[\E,K]) \cong \phi_\omega(\E/K)$.

We now show that $\phi_{\omega}(\E/K) \not \cong 0$.

Since $F \times_Q F \to F \times F$ is homologically $2$-connected
by part $(1)$ of the Fiber Product Theorem, it follows from
Stallings' Theorem that $F \times_Q F \to F \times F$ induces an
isomorphism of lower central series
quotients~\cite{Stallings:1965-1}. Thus, for all $k\geq 1$,
\[
\ker\left\{H_2(F\times_Q F) \to H_2\left(\frac{F \times
F}{\gamma_k(F \times F)}\right)\right\} = \phi_{k+1}(\E/K).
\]
Since $H_3(Q) \not \cong 0$, statement $(2)$ of the Fiber Product Theorem~\ref{theorem:FP} asserts that
\[
\ker\{H_2(F\times_Q F) \to H_2(F \times F)\} \not \cong 0.
\]
Hence, for all $k$
\[
\ker\{H_2(F\times_Q F) \to H_2(F \times F)\} \subset  \phi_k(\E/K).
\]
Hence $\phi_{\omega}(F \times_Q F) \not \cong 0$.

However, $\E/[\E,K] \to \E/K$ is a central extension by
construction, and since $F \times_Q F \subset F \times F$ is
residually nilpotent it follows that $\gamma_{\omega+1}(\E/[\E,K])
= \{1\}$.  This proves part $(3)$ of Proposition~\ref{prop:free
extension}.

Statement $(4)$ of Prposition~\ref{prop:free extension} follows
from Lemma~\ref{lemma:MP} as well, since $F \times F$ is
residually nilpotent and one easily computes the second equality
below.
\[
\gamma_{\omega}(\F/[\F,L])
 = \phi_{\omega}(F \times F) = 0\]

Finally, assume $Q$ is finite.  By the 1-2-3 Lemma of
G. Baumslag, M. Bridson, C. Miller III and H. Short
~\cite{Baumslag-Bridson-Martin-Miller}, it follows that $\E/K
\cong F \times_Q F$ is finitely presented.  So $H_2(F \times_Q F)
\cong K/[\E,K]$ is finitely presented, and therefore the group
$\E/[\E,K]$ in the following extension is finitely presented as
well.
\[
1 \to \frac{K}{[\E,K]} \to \frac{\E}{[\E, K]} \to \frac{\E}{K} \to
1
\]
\end{proof}

\section{Localization, closure, and relatively perfect subgroups}

We prove our main theorem.

\begin{theorem}\label{thm:main}
Let $Q$ be any finitely presented perfect group with a balanced presentation, such that $H_3(Q) \not \cong 0$.  (For instance, let $Q$ be the binary icosahedral group\footnote{A partial list of finite superperfect groups with non-trivial third homology group can be found at~\url{http://hamilton.nuigalway.ie/Hap/www/SideLinks/About/aboutSuperperfect.html}}.) Let $F \times_Q F \cong \E/K$ where $\E$ is free on the generating set of $Q$.  Let $G := \E/[\E,K]$ be the associated free central extension of $\E/K$.  Then
\begin{enumerate}
\item $G$ is a finitely generated group,
\item $\gamma_{\omega}(G)\neq \gamma_{\omega +1}(G)=\{1\}$,\item and if $E(G)$ represents either the Bousfield $H \Z$-localization of $G$, or if $E(G)$ is Levine's algebraic closure of $G$ and in
addition in this case, $Q$ is a finite group, then
\[
\gamma_{\omega}(G) = \ker\{G \to E(G)\}.
\]
In particular, in both cases, the $\ker\{G \to E(G)\}$ is
not a relatively perfect subgroup of $G$.
\end{enumerate}
\end{theorem}

\begin{proof}
All of the steps in this theorem follow from Proposition~\ref{prop:free extension}.

Observe that $\E/[\E,K]$ is finitely generated by Proposition~\ref{prop:free extension}, statement $(1)$.  Also, by statement $(2)$ of Proposition~\ref{prop:free extension}, the homomorphism
\[
\frac{\E}{[\E,K]} \to \frac{\F}{[\F,L]}
\]
is homologically $2$-connected.  $\gamma_{\omega}(G) \neq \gamma_{\omega+1}(G) = \{1\}$ by statement $(3)$ of Proposition~\ref{prop:free extension}.

We first prove the result for Bousfield $H \Z$-localization of groups.

Since $F \times_Q F \to F \times F$ is homologically
$2$-connected, this homomorphism induces an isomorphism of
$H\Z$-localizations, $E(F \times_Q F) \to E(F \times F)$.  By
statements $(3)$ and $(4)$ of Proposition~\ref{prop:free
extension}, $\gamma_{\omega}(G) = \gamma_{\omega}(\E/[\E,K]) \neq \{1\}$, and has
trivial image in $\F/[\F,L]$ by statement $(4)$.  But
$\gamma_{\omega}(\E/[\E,K])$ is not relatively perfect, by
statement $(3)$.

The proof for Levine's closure of groups needs only a small modification.

Since $Q$ is finite in this case, $\E/[\E,K]$ is finitely
presented (by statement $(1)$ of Proposition~\ref{prop:free
extension}.)  The group $\F/[\F,L]$ is finitely presented as well.
Since, {\em in addition}, $\E/[\E,K] \to \F/[\F,L]$ is
homologically $2$-connected, we get an isomorphism of closures,
$E(\E/[\E,K]) \to E(\F/[\F,L])$.  So $\gamma_{\omega}(K/[\E,K]) =
\ker\{\E/[\E,K] \to \F/[\F,L]\}$.  But $K/[\F,K]$ is not
relatively perfect in the sense of Levine, proving the Theorem.
\end{proof}

\bibliographystyle{amsalpha}

\bibliography{research}

\providecommand{\bysame}{\leavevmode\hbox to3em{\hrulefill}\thinspace}
\providecommand{\MR}{\relax\ifhmode\unskip\space\fi MR }
\providecommand{\MRhref}[2]{%
  \href{http://www.ams.org/mathscinet-getitem?mr=#1}{#2}
}
\providecommand{\href}[2]{#2}
\begin{thebibliography}{BBMIS00}

\bibitem[BBMIS00]{Baumslag-Bridson-Martin-Miller}
Gilbert Baumslag, Martin~R. Bridson, Charles~F. Miller~III, and Hamish Short,
  \emph{Fibre products, non-positive curvature, and decision problems},
  Comment. Math. Helv. \textbf{75} (2000), no.~3, 457--477. \MR{1793798
  (2001k:20091)}

\bibitem[Bou74]{Bousfield:1974-1}
A.~K. Bousfield, \emph{Homological localizations of spaces, groups, and {$\pi
  $}-modules}, Localization in group theory and homotopy theory, and related
  topics (Sympos., Battelle Seattle Res. Center, Seattle, Wash., 1974),
  Springer, Berlin, 1974, pp.~22--30. Lecture Notes in Math., Vol. 418.
  \MR{MR0370561 (51 \#6788)}

\bibitem[Bou75]{Bousfield:1975-1}
\bysame, \emph{The localization of spaces with respect to homology}, Topology
  \textbf{14} (1975), 133--150. \MR{MR0380779 (52 \#1676)}

\bibitem[Bou77]{Bousfield:1977}
\bysame, \emph{Homological localization towers for groups and {$\Pi
  $}-modules}, Mem. Amer. Math. Soc. \textbf{10} (1977), no.~186, vii+68.
  \MR{0447375 (56 \#5688)}

\bibitem[BR12]{Bridson-Reid:2012-1}
Martin.~R. Bridson and Alan.~W. Reid, \emph{Nilpotent completions of groups,
  {G}rothendieck pairs, and four problems of {B}aumslag}, arXiv:1211.0493,
  November 2012.

\bibitem[Cha08]{Cha:2004-1}
Jae~Choon Cha, \emph{Injectivity theorems and algebraic closures of groups with
  coefficients}, Proc. London Math. Soc. \textbf{96} (2008), no.~1, 227--250.

\bibitem[CO10]{Cha-Orr:2010-02}
Jae~Choon Cha and Kent~E. Orr, \emph{Hidden torsion and homology cobordism of
  $3$-manifolds}, To appear in Journal of Topology, 2010.

\bibitem[DFOS89]{Farjoun-Orr-Shelah:1989}
E.~Dror~Farjoun, K.~Orr, and S.~Shelah, \emph{Bousfield localization as an
  algebraic closure of groups}, Israel J. Math. \textbf{66} (1989), no.~1-3,
  143--153. \MR{MR1017158 (90j:55016)}

\bibitem[Dwy75]{Dwyer:1975-1}
William~G. Dwyer, \emph{Homology, {M}assey products and maps between groups},
  J. Pure Appl. Algebra \textbf{6} (1975), no.~2, 177--190. \MR{0385851 (52
  \#6710)}

\bibitem[EM08]{Emmanouil-Mikhailov:2008}
Ioannis Emmanouil and Roman Mikhailov, \emph{A limit approach to group
  homology}, J. Algebra \textbf{319} (2008), no.~4, 1450--1461. \MR{2383055
  (2008m:20086)}

\bibitem[Gut79]{Gutierrez:1979-1}
M.~A. Guti{\'e}rrez, \emph{Concordance and homotopy. {I}. {F}undamental group},
  Pacific J. Math. \textbf{82} (1979), no.~1, 75--91. \MR{549834 (82a:57020)}

\bibitem[Hec12]{Heck:2010-01}
Prudence Heck, \emph{Twisted homology cobordism invariants of knots in
  aspherical manifolds}, Int. Math. Res. Not. IMRN \textbf{2012} (2012),
  no.~15, 3434--3482.

\bibitem[LD88]{LeDimet:1988-1}
Jean-Yves Le~Dimet, \emph{Cobordisme d'enlacements de disques}, M\'em. Soc.
  Math. France (N.S.) (1988), no.~32, ii+92. \MR{90e:57046}

\bibitem[Lev88]{Levine:1988-1}
Jerome~P. Levine, \emph{Link concordance}, Algebra and Topology 1988 (Taej\u
  on, 1988), Korea Inst. Tech., Taej\u on, 1988, pp.~57--76. \MR{90k:57030}

\bibitem[Lev89a]{Levine:1989-1}
\bysame, \emph{Link concordance and algebraic closure. {I}{I}}, Invent. Math.
  \textbf{96} (1989), no.~3, 571--592. \MR{91g:57007}

\bibitem[Lev89b]{Levine:1989-2}
\bysame, \emph{Link concordance and algebraic closure of groups}, Comment.
  Math. Helv. \textbf{64} (1989), no.~2, 236--255. \MR{91a:57016}

\bibitem[Lev90]{Levine:1990-1}
J.~P. Levine, \emph{Algebraic closure of groups}, Combinatorial group theory
  ({C}ollege {P}ark, {MD}, 1988), Contemp. Math., vol. 109, Amer. Math. Soc.,
  Providence, RI, 1990, pp.~99--105. \MR{1076380}

\bibitem[Lev94]{Levine:1994-1}
Jerome~P. Levine, \emph{Link invariants via the eta invariant}, Comment. Math.
  Helv. \textbf{69} (1994), no.~1, 82--119. \MR{95a:57009}

\bibitem[MP06]{Mikhailov-Passi:2006-1}
Roman Mikhailov and Inder Bir~S. Passi, \emph{Faithfulness of certain modules
  and residual nilpotence of groups}, Internat. J. Algebra Comput. \textbf{16}
  (2006), no.~3, 525--539. \MR{2241621 (2008k:20076)}

\bibitem[RS04]{Rodriguez-Scevenels:2004-1}
Jos{\'e}~L. Rodr{\'{\i}}guez and Dirk Scevenels, \emph{Homology equivalences
  inducing an epimorphism on the fundamental group and {Q}uillen's plus
  construction}, Proc. Amer. Math. Soc. \textbf{132} (2004), no.~3, 891--898.
  \MR{2019970 (2005a:55006)}

\bibitem[Sak06]{Sakasai:2006-1}
Takuya Sakasai, \emph{Homology cylinders and the acyclic closure of a free
  group}, Algebr. Geom. Topol. \textbf{6} (2006), 603--631 (electronic).
  \MR{2220691 (2007m:57002)}

\bibitem[Sta65]{Stallings:1965-1}
John Stallings, \emph{Homology and central series of groups}, J. Algebra
  \textbf{2} (1965), 170--181. \MR{31 \#232}

\end{thebibliography}

\end{document}